\documentclass[11pt,reqno]{article}
\usepackage{amsmath,amsthm,amsfonts,amssymb,amscd}

\usepackage[latin1]{inputenc}

\usepackage{amssymb}
\theoremstyle{plain}
\newtheorem{Theorem}{Theorem}
\newtheorem{Corollary}{Corollary}

\newtheorem{Lemma}{Lemma}
\newtheorem{Proposition}{Proposition}

\theoremstyle{Definition}

\theoremstyle{Remark}
\newtheorem{Remark}{Remark}

\numberwithin{equation}{section}

\def\Z{\mathbb{Z}}                 
\def\N{\mathbb{N}}                   
\def\F{{\cal F}}                      

\def\fa{{\mathcal{F}}}
\def\O{{\mathcal{O}}}

\def\po{{\partial}}

\def\vr{{\varphi}}
\def\ga{{\gamma}}

\def\ov{\overline}
\def\al{{\alpha}}

\def\re{{\mathbb{R}}}

\def\bd{{\mathbb{D}}}
\def\bc{{\mathbb{C}}}

\begin{document}
\title{ $\mathbb{C^*}$- Actions on Stein analytic spaces  with isolated singularities}
\author{C. Camacho,  H. Movasati and B. Sc\'ardua}

\date{}
{\tiny
\maketitle
}


\section{Introduction}
\label{Section:intro}

Let $V$ be an irreducible complex analytic space of dimension two
with normal singularities and $\vr:\mathbb{C^*}\times V\to V$ a
holomorphic action of the group $\mathbb{C^*}$ on $V$. Denote by
$\fa_\vr$ the foliation on $V$ induced by $\vr$. The leaves of
this foliation are the one-dimensional orbits of $\vr$.
We will assume that there exists  a \emph{dicritical} singularity
$p\in V$ for the $\bc^*$-action, i.e. for some neighborhood $p\in
W\subset V$ there are infinitely many leaves of $\mathcal
{F}_\vr|_{W}$ accumulating only at $p$. The closure of such a local
leaf is an invariant local analytic curve called a \emph{separatrix}
of $\mathcal{F}_\vr$ through $p$. In \cite{Orlik} Orlik and Wagreich
studied the $2$-dimensional affine algebraic varieties embedded in
$\mathbb{C}^{n+1}$, with an isolated singularity at the origin, that
are invariant by an effective action of the form
\, $\sigma_Q(t,(z_{0},...,z_{n}))=(t^{q_{0}}z_{0},..., t^{q_{n}}z_{n})$
where $Q=(q_0,...,q_n) \in\mathbb N^{n+1}$, i.e. all $q_{i}$ are positive
integers. Such actions are called \emph{good} actions.
In particular they classified the algebraic surfaces embedded
in $\mathbb{C}^{3}$ endowed with such an action. It is easy to see
that any good action on a surface embedded in $\mathbb{C}^{n+1}$ has
a dicritical singularity at $0\in\mathbb{C}^{n+1}$.
Conversely, it is the purpose of this paper to show that good actions
are the models for analytic $\mathbb{C^*}$-actions on Stein analytic
spaces of dimension two with a dicritical singularity.
In this paper all spaces are connected and complex analytic.

\begin{Theorem}
\label{Theorem:main}
  Let $V$ be a normal Stein analytic space of
dimension two and $\vr$ a $\bc^*$-action on $V$ with at least  one
dicritical singularity $p\in V$. There is an embedding
$h:V\to\mathbb{C}^{n+1}$, for some $n$, onto an algebraic
subvariety ${\mathcal V}:=h(V)$ and a good action $\sigma_Q$ on
$\mathbb{C}^{n+1}$, leaving ${\mathcal V}$ invariant and
analytically conjugate to $\vr$, i.e.,
$$
h(\vr(t,x))=\sigma_Q(t,h(x)),\ \forall x\in V,\ t\in\bc^*.
$$
\end{Theorem}
Notice that this theorem implies that there is no other
singularity of $\vr$ apart from $p\in V$. The above theorem can be
considered as a GAGA principle for Stein varieties with $\bc^*$-actions.
This answers a question posed by some authors (see for instance the
comments after Proposition 1.1.3 in \cite{Orlik} and references there).

\begin{Corollary}
\label{Corollary:smooth} Let $V$ be a  smooth Stein surface
endowed with a $\bc^*$-action having a  dicritical singularity at
$p\in V$. Then $V$ is biholomorphic to  $\bc^2$.
\end{Corollary}

The proof of Theorem \ref{Theorem:main} will also provide a proof
of the following:

\begin{Theorem}
\label {Theorem:main2}
 The moduli space of pairs $(V,\vr),\ dim(V)=2$, with
at least  one dicritical
singularity for $\vr$  as in Theorem \ref{Theorem:main}, is
the following data
\begin{enumerate}
 \item
A Riemann surface $\sigma_0$ of genus $g$ and $s$-points
$r_1, r_2,\ldots, r_s$ on $\sigma_0$ considered up to the
automorphism group of $\sigma_0$.
\item
A line  bundle $L$ on $\sigma_0$ with $c(L)=-k\leq -1$.
\item
For each $i=1,2,\ldots,s$ a sequence of integers
$-k_{j}^i, \ j=1,2,\ldots, n_i,\ k_{j}^i\geq 2$, such that
$$
\sum_{i=1}^s\frac{1}{[k^i_{1}, k^i_{2}, \ldots, k^i_{n_i}]}<k,
$$
where

\[
[k_{1}^i,k_{2}^i,...,k_{n_i}^i]=k_{1}^i-\frac{1}{k_{2}^i-\frac{1}{\ddots}}.
\]

\end{enumerate}
Conversely,  $\it 1,2$ and $\it 3$ imply the existence
of a pair $(V,\vr)$.
\end{Theorem}
The above data can be read from the minimal resolution of the
desingularization at $p\in V$ of the foliation induced by $\vr$.

The proof of Theorem \ref{Theorem:main} consists of the following
 steps. We first analize in \S \ref{Section:resolution} the
resolution of the singularity $p\in V$ and obtain Theorem
\ref{Theorem:desingularization} which is an analytic version of
a theorem proved in  \cite{Orlik}. It turns out that there is
only one element $\sigma_0$ of arbitrary genus in the divisor of the
resolution of $p\in V$ on which $\mathbb{C}^*$ acts
transversely. All other divisors are Riemann spheres and are
invariant under the action of $\mathbb{C}^*$.
In \S \ref{Section:linearmodel} we linearize the
$\mathbb{C}^*$-action in a neighborhood of $\sigma_0$. The main
theorem of this section, Theorem \ref{mai}, does not require any
hypothesis on the self intersection number of $\sigma_0$.
In  \S \ref{Section:basin} we first introduce the linear model for
the resolution of $p\in V$ and then extend the linearization
obtained in the previous section to the basin of attraction of
$p\in V$.
In \S \ref{Section:basins} we prove that the basin of attraction
of $p\in V$ is the whole space $V$ and so the constructed
linearization provides the conjugacy claimed in Theorem
\ref{Theorem:main}.


\section{Resolution of singularities}
\label{Section:resolution}
In order to prove Theorem \ref{Theorem:main} we first describe the
resolution of the action $\vr$ and then compare it with the
resolution of a model good action.
\subsection{Holomorphic foliations}
We start with the resolution theorem for normal two dimensional
singularities (see \cite{Laufer}) and the resolution theorem for
holomorphic foliations (see \cite{Seidenberg}, \cite{Camacho-Sad})
that combined together assert, first, that there exists a proper
holomorphic map $\rho:\tilde V\to V$ such that $D$:=$\rho^{-1}(p)
=\bigcup_{i=0}^{r}\sigma_{i}$, is a finite union of compact Riemann
surfaces $\sigma_{i}$ intersecting at most pairwise at normal
crossing points, and then that $\tilde V$ is an analytic space of
dimension two with no singularities near $D$. More precisely, the
$\sigma_{i}$'s are compact Riemann surfaces without singularities
such that if $\sigma_{i}\cap \sigma_{j}\ne\emptyset$ then
$\sigma_{i}$ and $\sigma_{j}$ have normal crossing and $\sigma_{i}
\cap\sigma_{j}\cap\sigma_{k}=\emptyset$ if $i\neq j\neq k \neq i$.
Moreover, the {\it intersection matrix} $(\sigma_i\cdot \sigma_j)$
is negative definite (\cite{Laufer})and the restriction of $\rho$ to
$\tilde V\backslash D$ is a biholomorphism onto $V\backslash \{p\}$.
By means of this restriction $\mathcal{F}_\vr$ induces a foliation
$\mathcal{\tilde F}_\vr$ on $\tilde V\backslash D$
that can be extended to $\tilde V$ as a foliation with isolated
singularities. Each one of these
singularities can be written in local coordinates $(x,y)$ around
$0\in{\mathbb{C}^2}$ in one of the following forms : $(i)\,\emph
{simple singularities}$: $ xdy-y (\mu+\cdots)dx=0$ , $ \mu\notin \mathbb
{Q}_{+}$,  where the points denote higher order terms; $(ii)\emph
{ saddle-node singularities}$: $x^{m+1}dy- (y+a x^{m}y+\cdots)dx=0$,
$a\in\mathbb C$, $m\in\mathbb N$. A simple singularity has two
invariant manifolds crossing normally, they correspond to the $x$
and $y$-axes. The saddle-node has an invariant manifold corresponding
to the $y$-axis and, depending on the higher order terms, it may
not have another invariant curve (see \cite{Mr-Rm2}). The resolution
of $\mathcal{F}_\vr$ can be obtained in such a way that the
elements $\sigma_{i}$ fall in two categories. Either $\sigma_{i}$
is a \emph{dicritical component}, when $\mathcal {\tilde F}_\vr$
is everywhere transverse to $\sigma_{i}$, or a
\emph{nondicritical} component when $\sigma_{i}$ is tangent to
$\mathcal{\tilde F}_\vr$. In a similar way, by means of the
restriction $\rho$ to $\tilde V\backslash D$ the $\mathbb{C^*}$-
action $\vr$ on $V\backslash \{p\}$ induces a $\mathbb{C^*}$-
action $\tilde \vr$ on $\tilde V\backslash D$ that can be extended
to $D$ as a $\mathbb{C^*}$- action (see \cite{Orlik2}). For this it
is enough to observe that $D\subset \tilde V$ is analytic of
codimension one, $\tilde V$ is a normal analytic space and
$\tilde\vr$ is bounded in a neighborhood of $D$. We have therefore
that the orbits of $\tilde \vr$ are contained in the leaves of the
foliation $\tilde \fa_\vr$.

The divisor $D$ forms a graph with vertices $\sigma_{i}$ and sides
the nonempty intersections $\sigma_{i}\cap \sigma_{j}$.
A \emph{star} is a contractible connected graph where at most one
vertex, called its \emph{center}, is connected with more than two
other vertices. A \emph{weighted graph} is a graph where at each
vertex is associated its genus and its self-intersection number.

\subsection{On a theorem of Orlik and Wagreich }
In this section we describe the resolution of $p$ as a singular point
of $V$ and as a singularity of $\fa_\vr$.

\begin{Theorem}
\label{Theorem:desingularization} Let $V$ be a normal Stein
analytic space of dimension two and $\vr$ a $\bc^*$-action on $V$
with a dicritical singularity at $p\in V$. Then there is a
resolution $\rho: \tilde V\to V$   of $\fa_\vr$ at the point $p\in V$
such that
\begin{enumerate}
 \item
$\rho^{-1}(p)=\bigcup_{i=0}^{r}\sigma_{i}$ is a
weighted star graph centered at the Riemann surface $\sigma_0$ of
genus $g$, and consisting of Riemann spheres $\sigma_{i}$, $i>0$;
\item
$\sigma_{0}$ is the unique dicritical component of
$\mathcal{\tilde F}_\vr =\rho^*\mathcal{F}_\vr$;
\item
the pull-back action $\tilde \vr$ on $\tilde V$ is trivial on
$\sigma_{0}$ and nontrivial on each $\sigma_{i}$, $i>0$;
\item
The singular points of
$\mathcal{\tilde F}_\vr$ are $\sigma_i\cap \sigma_j\ne\emptyset,\ i,j\not
= 0$ and all of them are simple.
\end{enumerate}
\end{Theorem}
In the algebraic context in which $V$ is affine and the
$\bc^*$-action is algebraic, the above theorem with items
$\it{1,2}$ and $\it{3}$
is a result of Orlik and Wagreich (see \cite{Orlik}). Our proof uses
the theory of holomorphic foliations on complex manifolds instead of
topological methods.
In order to prove Theorem \ref{Theorem:desingularization} we need the
following index theorem.
\subsection{The Index theorem}
Let $\sigma$ be a Riemann surface embedded in a two dimensional
manifold $S$ ; $\fa$ a foliation on $S$ which leaves $\sigma$
invariant and $q\in \sigma$. There is a neighborhood of $q$ where
$\sigma$ can be expressed by $(f=0)$ and $\fa$ is induced by the
holomorphic 1-form $\omega$ written as
$\omega=hdf+f\eta$.
Then we can associate the following index:
$$
i_{q}(\fa, \sigma):=-{\rm Residue}_q(\frac{\eta}{h})|_{\sigma}
$$
relative to the invariant submanifold $\sigma$.
In the case of a simple singularity as defined above if $\sigma$
is locally $(y=0)$ and $q=0$, this index is equal to $\mu$
(quotient of eigenvalues). In the case of a saddle-node, if
$ \sigma $ is equal to $(x=0)$ and $q=0$, this index is zero.
At a regular point $q$ of $\fa$ the index is zero.
The index theorem of \cite{Camacho-Sad} asserts that the sum of
all the indices at the points in $\sigma$ is equal to the
self-intersection number $\sigma\cdot\sigma$:
$$
\sum_{q\in{\sigma}} i_q(\fa,\sigma)=\sigma\cdot\sigma.
$$
\subsection{Proof of Theorem \ref{Theorem:desingularization}}
 By hypothesis, in the resolution of $p\in V$ there is at least
one dicritical component, say $\sigma_{0}$. Then the action $\tilde
\vr$ extends to $\sigma_0$ as a set of fixed points. We claim that
$\sigma_0$ is the unique dicritical component. Indeed, at each
dicritical component the $\mathbb{C^*}$- action $\tilde \varphi$
is trivial. Since $V$ is normal at $p\in V$, $\rho^{-1}(p)$ is
connected (\cite{Laufer}), thus if there is another dicritical
component, say $\sigma_{i}$, then there would exists
$\mathbb{C^*}$- orbits of $\tilde \varphi$, with compact analytic
closure  crossing $\sigma_{0}$ and $\sigma_{i}$ transversely
contradicting the fact that $V$ is Stein. Thus $\sigma_{0}$ is the
only dicritical component, and the action $\tilde \varphi$ is
trivial on $\sigma_{0}$. The same argument shows that there cannot
be cycles of components of $D$ starting and ending at $\sigma_0$.
Thus the graph associated to $\rho$ is contractible. \\

A
\emph{linear chain} at a point $q\in \sigma_{0}$ is a union of
compact Riemann surfaces, elements of the divisor $D$, say
$\sigma_{1},..., \sigma_{n}$ such that $\sigma_{1}\cap \sigma
_{0}=\{q\}$ and $\sigma_{i}\cap \sigma_{j}$ is nonempty if and
only if $i=j-1$ and in this case
it is a point, for $j= 2,...,n$. \\

\begin{Lemma}
\label{Lemma:desingularization} Suppose that ${r_{1}, r_{2},...,
r_{s}}$ are the crossing points at $\sigma_{0}$ of the divisor
$D$. Then the divisor $D$ consists of the union of $\sigma_{0}$
and linear chains of Riemann spheres at each of these crossing
points.
\end{Lemma}

\begin{proof}
Consider the divisor $D$ at the point $r_{1}$ renamed as
$p_{0}$. Let $\sigma_{1}$ be such that $p_{0}=\sigma_{0}\cap
\sigma_{1}$. We claim that the $\mathbb{C^*}$-action $\tilde \varphi$
on $\sigma_{1}$ is nontrivial with a fixed point at $p_{0}$.
Indeed it can be represented in local coordinates $(x,y)$, where
$(x=0)=\sigma_0, \, (y=0)=\sigma_1$, by the vector field
$Y=(Y_1,0)$ with $Y_1(0,y)=0$. Consider the restriction of the
action $\tilde \vr$ to the subgroup $S^1\subset \bc^*$. Then in the
$\bc$-plane $(y=y_0)$ the $S^1$-orbit of a generic point $(x,y_0),
\, x\ne 0$, will turn $l$ times around $(0,y_0)$ and this number,
 which is different from zero, will be constant as $y_0\to 0$.
Therefore $\tilde \vr$ extends to the $x$-axis $\sigma_1$ as a
nontrivial $\bc^*$-action. Therefore $\sigma_{1}$ is a Riemann
sphere and there is another point $p_{1}\in \sigma_{1}$ which is
fixed by $\tilde \varphi$. Since $p_{1}$ is the unique singularity of
$\mathcal{\tilde F_{\vr}}$ in $\sigma_{1}$ we must have that the index of
$\tilde \fa_{\vr}$ with respect to the invariant manifold $\sigma_1$
at $p_1$ is given by (\cite{Camacho-Sad})
\[i_{p_{1}}(\mathcal{\tilde F_{\vr}}, \sigma_{1})=\sigma_{1}.\sigma_{1}=
-k_{1}, \, \,   k_{1}\in \mathbb{N}. \]

Therefore $p_{1}$ cannot be a saddle-node, as in this case this
index would be zero. This implies that $p_{1}$ is simple for
$\mathcal{\tilde F_{\vr}}$. Either the chain ends at
$\sigma_{1}$ or there is another component, say $\sigma_{2}$, such
that $\{p_{1}\}= \sigma_{1}\cap \sigma_{2}$. In this last case,
$p_{1}$ is simple. We claim that the action $\tilde \varphi$ on
$\sigma_{2}$ is nontrivial. Indeed, let $(x,y)$ be a system
of coordinates in a neighborhood $\mathcal N$ of $p_1=(0,0)$
such that $(x=0)=\sigma_{1}\cap\mathcal N$,
$(y=0)=\sigma_{2}\cap\mathcal N$. By derivation along the parameter
of the group, the action $\vr$ induces a vector field $Y$ on
$\mathcal N$. Assuming by contradiction that $\vr$ is trivial on
$\sigma_{2}$ we would have $ Y(x,0)=0$ and we can assume, changing
coordinates if necessary, that $ DY(x,0)= diag(0,\lambda_{x})$,
$\lambda_{0}\neq 0$. By continuity, $\lambda_{x}\neq 0$ for $x$
small enough. By the invariant manifold theorem for ordinary
differential equations, there is a fibration invariant by $Y$,
 transverse to $\sigma_{2}$, whose fibers are the subsets of $\mathcal N$
defined as
$\tau_{x}=\{(x,y); \lim_{t\to 0}\vr(t,(x,y))=(x,0)\}$, $\tau_0=\sigma_1$.
Thus $\sigma_{2}$ is a dicritical component of $\mathcal{\tilde F_{\vr}}$,
which is a contradiction.
Therefore $\sigma_{2}$ will be a {\sl Riemann sphere} with another fixed
point $p_{2}\in \sigma_{2}$ for the action $\tilde \varphi$. It is
clear that the corresponding index will be given by
\[
i_{p_{2}}(\mathcal{\tilde F_{\vr}},\sigma_{2})= -k_{2}+1/k_{1}\ne 0,  \, \,
 k_2=-\sigma_{2}.\sigma_{2}\in \mathbb N. \]

More generally, the linear chain will consist of a finite
sequence of elements of the divisor $\sigma_{0}, \sigma_{1},...,
\sigma_{n}$ such that $\sigma_{i}$, for $i\ne 0$, is a Riemann
sphere where the action $\tilde \varphi$ is nontrivial, and
$\sigma_{i}\cap\sigma_{i+1}= \{p_{i}\}$ is a simple singularity of
$\mathcal{\tilde F_{\vr}}$ for $i=1,..., n-1$. Denote by $-k_{i}=\sigma_{i}.
\sigma_{i},  k_{i}\in \mathbb N$. At each point $p_{i}$ the
index of this singularity relative to $\sigma_{n}$ is
\[
i_{p_j}(\tilde \fa_{\vr},\sigma_j)=-[k_{j},k_{j-1},...,k_{1}],\] where
we have a continued fraction
\[
[k_{j},k_{j-1},...,k_{1}]=k_{j}-\frac{1}{k_{j-1}-\frac{1}{\ddots}}.
\]
We claim that the numbers $[k_{j},k_{j-1},...,k_{1}]$, $j=1,...,n,$
are all well defined and different from zero. Indeed, this is a
consequence of the fact that the intersection matrix
$(\sigma_{i}\cdot\sigma_{i})$ is negative definite
(\cite{Laufer}). Let $M$ be a real symmetric $n\times n$ matrix and
Q a non-singular real $n\times n$ matrix. Then $M$ is negative
definite if and only if $Q^{t}MQ$ is negative definite. Given the
matrix $M=(\sigma_{i}\cdot\sigma_{j})$ we take $Q$ as the matrix
with one's in the diagonal, $a$ in the $(1,2)$ entry, and zeros
elsewhere. Then a convenient choice of $a$ will yield a matrix
$Q^{t}MQ$ with $-k_{1}$ in the $(1,1) $ entry and zeros in the
$(1,2)$ and $(2,1)$ entries. Repeating
this procedure we obtain that the following diagonal matrix
$$
{\rm diag}(-k_{1}, -[k_{2}, k_{1}],..., -[k_{n},k_{n-1},...,k_{1}])
$$
is negative definite, proving the claim and the lemma.

\end{proof}

Theorem~\ref{Theorem:desingularization} follows from the above
discussion and Lemma~\ref{Lemma:desingularization}.

\section{Linearization around the dicritical divisor
}
\label{Section:linearmodel}

Let ${\mathbb D}=\{z\in{\mathbb C} \mid |z|<1\}$ be the unit
disk. In the previous section we saw that the multiplicative
pseudo group $\mathcal G=({\mathbb C},{\mathbb D})-\{0\}$ acts
on $(\tilde V,\sigma_0)$ and the flow of the action $\vr$ is
transverse to $\sigma_0$. The purpose of this section is to
show that such an action is biholomorphically conjugated with
the canonical $\mathcal G$-action on the normal bundle to
$\sigma_0$ in $\tilde V$.
\subsection{$\mathcal G$-transverse actions to a Riemann surface}
 Let $\sigma$ be a Riemann surface embedded in a surface $S$.
We say that $\psi$ is a transverse  $\mathcal G$-action on $(S,\sigma)$
if
\begin{enumerate}
\item For all $a\in \sigma$ and $t\in \mathcal G$ we have $\psi(t,a)=a$.
\item There is a foliation $\F$ on $(S,\sigma)$, transverse to
$\sigma$ such that each leaf of $\F$ is the closure of
$\{\psi(t, a)\mid t\in \mathcal {G} \}$ for some $a\in (S,\sigma)-\sigma$.
\end{enumerate}
A typical example of a $\mathcal G$-action is the following:
We consider a line bundle $L$ on $\sigma$ and the
embedding $\sigma\hookrightarrow L$ given by the zero section. Now for
every $q\in \mathbb N$ we have a transverse  $\mathcal G$-action on $(L,\sigma)$
given by $(t,a)\mapsto t^qa$. It turns out that up to biholomorphy these
are the only transverse $\mathcal G$-actions.
\begin{Theorem}[Linearization theorem]
\label{mai} Let $\sigma$ be a Riemann surface embedded in a surface $S$
and $\psi$ a transverse $\mathcal G$-action on $(S,\sigma)$. Then $\psi$
is linearizable in the sense that there exist a biholomorphism
$h:(S,\sigma)\to (N,\sigma)$, where $N$ is the normal bundle to $\sigma$
in $S$, and a natural number $q$ such that
$h(\psi(t,a))=t^qh(a)$ for any $a\in (S,\sigma)$.
\end{Theorem}
Notice that the linearization of $\psi$ yields also the linearization of the
associated foliation. An immediate corollary of the above theorem is that
non-linearizable neighborhoods do not admit any transversal $\mathcal G$-action.
For instance, Arnold's example in which $\sigma$ is a torus of self-intersection
number zero in some complex manifold of dimension two is not linearizable and
so it does not admit any transversal $\mathcal G$-action (see \cite{arn}).

\subsection{Local linearization} Let $S=(\mathbb {C}^2,0)$ and
$0\in \sigma\subset S$ be a smooth curve in $S$. In a similar way as
before we define a $\mathcal G$-action on $(S,\sigma)$ transverse to
$\sigma$ and call it the local transverse $\mathcal G$-action.
\begin{Lemma}
\label{Lemma:26mar07} Any  local transverse $\mathcal G$-action  can
be written in a local system of coordinates in the form
$
\psi(t, (x,y))=(x,t^qy).
$
\end{Lemma}
\begin{proof}
We take a coordinates system $(x,y)$ around $0\in\bc^2$ such that the the
foliation $\fa_\psi$ is given by $dx=0$ and $\sigma$ is given
by $y=0$. In these coordinates the flow $\psi_t$ of the $\bc^*$-action is
given by:
$$
\psi_t:(\bc^{2},0)\rightarrow (\bc^{2},0),\  \psi_t(x,y)= (x, p_{t,x}(y)).
$$
Since the orbits of $\psi$ tend to $\sigma$ when $t$ tends to zero, $p_{t,x}$
is a holomorphic function in $t\in (\bc,\mathbb{D})$. We have also $p_{t,x}(0)=0$
because $\sigma$ is the set of fixed points of $\psi$. We can write $p_{t,x}(y)$
as a series
$$
p_{t,x}(y)=\sum_{i=1}p_i(t,x)y^i.\ 
$$
Substituting the above term in $\psi(t_1t_2,a)=\psi(t_1,\psi(t_2,a))$
we obtain
$$
p_1(t_1t_2,x)=p_1(t_1,x)p_1(t_2,x), \ t_1,t_2\in \mathcal{G}, \ x\in(\bc,0 ).
$$
Since $p_1$ is holomorphic at $t=0$, the derivation of the above equality
in $t_1$ implies that $p_1(t,x)=t^q$ for some $q\in \N$. Now, by the Theorem
on the linearization of germs of holomorphic mapings, there is a unique
$f_{t,x}:(\bc,0)\mapsto(\bc,0)$
which is tangent to the identity, depends holomorphically on $t,x$ and
$$
f^{-1}_{t,x}\circ p_{t,x} \circ f_{t,x}(y)=t^qy.
$$
The $\bc^*$-action $\psi$ in the coordinates $(\tilde x,\tilde y)=(x,f_{t,x}(y))$
has the desired form.
\end{proof}

Now consider on $S$ a foliation $\F$ which is transverse to $\sigma$
(no $\mathcal G$-action is considered). Let $\omega$ be a $1$-form on
$S$ such that
$$
div(\omega)=\sigma+nL_{0},
$$
where $n\in \Z$ and $L_0$ is the leaf of $\F$ through $0\in S$.
\begin{Lemma}
\label{28mar07} Given a local system of coordinates $x$ in $\sigma$,
there is a unique system of coordinates $(\tilde x,\tilde y)$ in
$S$ such that
\begin{enumerate}
\item The restriction of $\tilde x$ to $\sigma$ is $x$; \item The
$1$-form $\omega $ in $(\tilde x,\tilde y)$ is of the form $\tilde
x^n\tilde yd\tilde x$.
\end{enumerate}

\end{Lemma}
\begin{proof}
For the proof of the existence we take a coordinates system $(\tilde
x,\tilde y)$ in a neighborhood of $0$ in $S$ such that $\sigma$ and
$\F$ in this coordinate system are given respectively by $\tilde
y=0$ and $d\tilde x=0$ and $\tilde x \mid_\sigma=x$. We write $\omega
=p\tilde x^n\tilde y d\tilde x$, where $p\in\O_{S},\ p(0)\not =0$.
By changing the coordinates $(\tilde x,\tilde y)\rightarrow
(\tilde x,p\tilde y)$ we obtain the desired coordinate system.
The uniqueness follows from the fact that  any local biholomorphism
$f:({\mathbb C}^2,0)\to ({\mathbb C}^2,0)$ which is the identity in
$\tilde y=0$ and $f^{*}\tilde x^n\tilde yd\tilde x=\tilde x^n\tilde
yd\tilde x$  is the identity map.
\end{proof}
\subsection{Construction of differential forms} Consider a Riemann
surface $\sigma$ embedded in a two dimensional manifold $S$. We take a
meromorphic section $s$ of the normal bundle $N$ of $\sigma$ in $S$ and
set
$$
div(s)=\sum n_ip_i,\ n_i\in\Z,\ p_i\in \sigma.
$$
\begin{Lemma}
\label{Lemma:daava} For a transverse  $\mathcal G$-action $\psi$ on
$(S,\sigma)$, there is a meromorphic function $u$ on $(S,\sigma)$ such
that
\begin{enumerate}
\item
$$
div(u)=\sigma-\sum n_ip_i,\ n_i\in\Z,\ p_i\in \sigma,
$$
\item
$$
u(\psi(t, a))=t^qu(a),\ a\in (S,\sigma),\ t\in \mathcal G.
$$
\end{enumerate}
Let $\tilde{v}$ be an arbitrary meromorphic function on $\sigma$ and
$v$ its extension to S along the foliation $\F$. The 1-form
\[
\omega=udv
\]
has the properties:
\begin{enumerate}
\item $\omega$ induces the foliation $\F$; \item The divisor of
$\omega$ is $\sigma+K$, where $K$ is $\F$-invariant.

\item
$
\psi_{t}^{*} \omega =t^q\omega,\ t\in \mathcal G,
$
where $\psi_{t}(x)=\psi(t,x)$.
\end{enumerate}

\end{Lemma}
\begin{proof}
In a local coordinate system  $(x_\alpha,y_\alpha)$ in a neighborhood
$U_\alpha$ of a point $p_\alpha$ of $\sigma$ in $S$ one can write the
$\mathcal G$-action as follows
$$
\psi (t, (x_\alpha,y_\alpha))=(x_\alpha,t^qy_\alpha),
$$
where $\sigma\cap U_\alpha=\{y_\alpha=0\}$. The meromorphic function
$u_\alpha=x^{-n}_\alpha y_\alpha $, where $n=n_i$ if $p=p_i$ for some
$i$ and $n=0$ otherwise, satisfies the conditions $\emph 1,\emph 2$ in
$U_\alpha$. We define $u_{\alpha\beta}:=\frac{u_\alpha}{u_\beta}$.
Now $L:=\{u_{\alpha\beta}\}\in H^1(S,\pi^{-1}
\O_\sigma^*)=H^1(\sigma,\O_\sigma^*)$, where $\pi :S\to \sigma$ is the
projection along the fibers. On the other hand, the line bundle associated
to $\sigma$ in $S$ and then restricted to $\sigma$ is the normal bundle of
$\sigma$ in $S$ and so by definition $L$ restricted to $\sigma$ is the trivial
bundle. This means that there are $a_\alpha\in\pi^{-1}\O^*_\sigma(U_\alpha)$
such that
$u_{\alpha\beta}=\frac{a_\alpha}{a_\beta}$.
Now, $\frac{u_\alpha}{a_\alpha}$ define  a meromorphic function on $S$
with the desired properties.
\end{proof}

\begin{Remark}\rm
In the case in which we have a transverse foliation $\F$ without
any transverse $\psi$ action, the linearization of $\F$
requires $\sigma\cdot \sigma< min(2-2g,0)$, where $g$ is the genus of
$\sigma$ (see \cite{Camachoetal, camo03}). In this case, in order to
construct $u$ with the first property we used this hypothesis
and proved that the restriction map ${\rm Pic}(X)\to{\rm Pic }(\sigma)$
is injective.  As we saw in the proof of Lemma~\ref{Lemma:daava},
in the presence of a transverse $\mathcal G$-action we do not need
any hypothesis on $\sigma\cdot\sigma$.

\end{Remark}
\subsection{Holomorphic equivalence of neighborhoods}

Now we consider two embeddings of $\sigma$ with transverse foliations.
\begin{Lemma}
\label{manooo} Let $\sigma$ be a Riemann surface embedded in two
surfaces $S_i,\ i=1,2$ and let $\F_i$ be a foliation transverse to
$\sigma$ on $S_i$ induced by a $1$-form $\omega_i$ such that the divisor
of $\omega_i$ is $\sigma+K_i$, where $K_i$ is $\F_i$-invariant and $K_1$
and $K_2$ restricted to $\sigma$ coincide. Then there is a unique
biholomorphism $h:(S_1,\sigma)\to (S_2,\sigma)$ such that
$h^{*}\omega_2=\omega_1$.
\end{Lemma}
\begin{proof}
Using Lemma \ref{28mar07}  we conclude that for a point $a\in \sigma$
there is a unique $h:(S_1,\sigma,a)\to (S_2,\sigma,a)$ such that $h$
restricted to $\sigma$ is the identity map and
$h^{*}\omega_2=\omega_1$. The uniqueness implies that these local
biholomorphisms coincide in their common domains and so they give
us a global biholomorphism $h:(S_1,\sigma)\to (S_2,\sigma)$ with the
desired property.




\end{proof}
\subsection{Proof of the linearization theorem}
Let us now prove Theorem~\ref{mai}. Take $i=1,2$ . Let $\sigma$ be a
Riemann surface embedded in two surfaces $S_i$ and let $\psi_i$ be a
transverse $\mathcal G$-action on $(S_i,\sigma)$ with the multiplicity
$q$ and corresponding foliation $\F_i$.
By Lemma \ref{Lemma:daava} we can construct a  $1$-form $\omega_i$ with
the properties $\emph{1, 2, 3}$. By construction of $\omega_i$, if
$div(\omega_i)=\sigma+K_i$ then $K_i$ restricted to $\sigma$ depends
only on $\tilde v$ and $s$ and so we can take the $K_i$'s so that
$K_1\mid_\sigma=K_2\mid_\sigma$. Now Lemma 5 implies that
there is a unique biholomorphism $h:(S_1,\sigma)\to (S_2,\sigma)$ such
that $h^{*}\omega_2=\omega_1$. We claim that $h$ conjugates also the
$\psi_i$'s. Fix $t\in \mathcal G$ and let $\psi_{i,t}:(S_i,\sigma)\to
(S_i,\sigma)$ be a biholomorphism defined by
$$
\psi_{i,t}(a):=\psi_i(t, a),\ a\in (S_i,\sigma).
$$
We have
$$
h^{*}\psi_{2,t}^{*}\omega_2=h^{*}t^q\omega_2=t^q\omega_1=
\psi_{1,t}^{*}\omega_1=
\psi_{1,t}^{*}h^{*}\omega_2.
$$
Since by Lemma \ref{manooo} the sole $f:(S_2,\sigma)\to(S_2,\sigma)$ such
that $f^{*}\omega_2=\omega_2$ is the identity map, we conclude that
$h^{*}\psi_{2,t}^{*}=\psi_{1,t}^{*}h^{*}$ and so
$h(\psi_1(t,a))=\psi_2(t, h(a))$.
\section{Linearization in the attraction basin}
\label{Section:basin}

In this section we associate to the foliation $\tilde \fa_\vr$ a
{\it linear model} and prove a linearization result based on the
existence of the $\mathcal G$-action transverse to $\sigma_0$.

\subsection{The linear model}
We can associate to the pair $(\tilde\fa_\vr,\tilde V)$
a linear model constructed as follows. Let $L$ be the normal bundle
of
$\sigma_0$ in $\tilde V$.
 We denote by $L^{-1}$ the dual of $L$. We can glue $L$ and
$L^{-1}$ together and obtain a compact projective manifold $\bar
L$ in the following way: Let $\{U_{\alpha}\}_{\alpha\in I}$ be an
open covering of $\sigma_0$ and $z_\alpha$ (resp. $z_\alpha'$) a
holomorphic without zero section of $L$  (resp. $L^{-1}$) on
$U_\alpha$. Then
$$
z_\alpha=g_{\alpha\beta}z_{\beta},\
z_\alpha'=g_{\alpha\beta}^{-1} z_{\beta}',\
L=\{g_{\alpha\beta}\}_{\alpha,\beta\in I} \in H^1(S,\O^*).
$$
For a point $a\in L_p, p\in U_{\alpha},\ a\not =0_p$ we define the
point $\frac{1}{a}\in L^{-1}_p$ by setting
$$
\frac{1}{a}=\frac{z_\alpha(p)}{a}z_\alpha'(p).
$$
The map $a\rightarrow 1/a$ does not depend on the chart $U_\alpha$
and gives us a biholomorphism between
$L-\sigma_0$ and $L^{-1}-{\sigma_\infty}$,
where $\sigma_0$ (resp. $\sigma_\infty$) is the zero section of $L$
(resp. $L^{-1}$).

For each point $r_i^{0}=r_i\in\sigma_0, \ i=1,2,\ldots,s$ we denote
by $r_i^\infty$ the unique intersection point of $\sigma_\infty$
and  $\bar{L}_{r_i^0}$. By various blow ups starting from
$r_i^\infty$  in the chain $\sigma_0, \bar{L}_{r_i^0},\sigma_\infty$,
we can  create a chain of divisors
$$
\sigma_0, \sigma_1^i, \sigma_2^i,\ldots,\sigma_{n_i}^i, \tilde \sigma,
 \tau_{m_i}^i, \tau_{m_i-1}^i, \ldots, \tau_{1}^i, \sigma_\infty
$$
such that
$$
\sigma_j^i\cdot \sigma_j^i=-k^i_j, \ j=1,2,\ldots,n_i,
\ \ \tilde \sigma\cdot\tilde \sigma=-1,
\ \ \ -l_{j}^i:=\tau^i_j\cdot \tau^i_j<-1, j=1,2,\ldots, m_i.
$$
The chain of self-intersections of the divisors in the blow-up process
is given by:
$$
(-k,0,k), (-k,-1,-1,k-1), (-k,-2,-1,-2,k-1),\ldots, (-k,  -k_1^i,-1,
\underbrace{-2,\cdots, -2}_{k_1^i-1 \hbox{ times }}, k-1)
$$
$$
(-k,-k_1^i,-2,-1,-3, \underbrace{-2,\cdots,-2}_{k_1^i-2
\hbox{ times }}, k-1), \cdots, (-k,-k_1^i,-k_2^i, \cdots,
-k_{n_i}^i,-1, l_{m_i}^i,\cdots, l_{2}^i,l_{1}^i, k-1).
$$
Repeating this construction at each point $r_i, \, i=1,...,s$ we
obtain a surface $X$. Let
$$
D_\infty=\sigma_\infty+\sum_{i=1}^s\sum_{j=1}^{m_i}
\tau_j^i, \
D_0=\sigma_0+\sum_{i=1}^s\sum_{j=1}^{n_i}\sigma_j^i.
$$
Now, $\tilde {\cal V}:= X-D_\infty$ is the desired linear model
variety. In $\bar L$ we have a canonical $\bc^*$ action whose
orbits are the fibers of $L$. It gives us a $\bc^*$-action
$\tilde \lambda$ on $\tilde{\mathcal V}$.
 We denote by $\tilde{\mathcal F}_\lambda$ the associated
foliation on $\tilde {\mathcal V}$. The pair
$(\tilde{\mathcal V}, \tilde {\mathcal F}_\lambda)$
will be called the {\it linear approximation} of
$(\tilde V,\tilde \fa_\vr)$.

In order to proceed with our discussion we need some definitions:
A divisor $Y=\sum_{i=1}^lY_i$  in a two-dimensional surface $X$
is a support of a divisor with positive (resp. negative) normal
bundle if there is a divisor $\tilde Y:=\sum_{i=1}^l a_iY_i $,
where the $a_i,\ i=1,2,\ldots,l$ are positive integers, such that
$\tilde Y\cdot Y_j>0( \text{ resp. }  <0)$, for $  j=1,\ldots,l$.

We say that the normal bundle of the divisor $\tilde Y$ in $X$
is positive (resp. negative). Observe that the normal bundle $N$ of
a divisor is positive (resp. negative) if and only if $N$ restricted
to each irreducible component of the divisor is positive
(resp. negative) (see \cite{Hartshorne} Proposition 4.3).
In fact the above number is the Chern class of $N\mid_{Y_i}$
(see \cite{Laufer} p. 62).

\begin{Lemma}
\label{23.4.07}
The following assertions are equivalent:
\begin{enumerate}
\item
The divisor $D_\infty$ is a support of divisor with positive
normal bundle.
\item
The self-intersection matrix of $D_0$ is negative definite.
\item
$$
\sum_{i=1}^s\frac{1}{[k^i_{1}, k^i_{2}, \ldots, k^i_{n_i}]}<k.
$$
\item
$$
\sum_{i=1}^s\frac{1}{[l^i_{1}, l^i_{2}, \ldots,l^i_{m_i}]}>s-k.
$$
\end{enumerate}
\end{Lemma}
\begin{proof}
$\it 1\Rightarrow 2$. From \cite{Hartshorne} Theorem 4.2 it follows
that one can make a blow down of the  divisor $D_0$ and so the self
intersection matrix of $D_0$ is negative definite.

$\it 2\Rightarrow 3$.  We remark that the diagonalization of the
intersection matrix of $D_0$ by the procedure given in Lemma 1
leads to
$$
{\rm diag}(\ldots, -k_{n_i}^i, -[k_{n_i-1}^i, k_{n_i}^i],...,
-[k_{1}^i,k_{2}^i,...,k_{n_i}^i], \ldots, -k+ \sum_{i=1}^s
\frac{1}{[k^i_{1}, k^i_{2}, \ldots, k^i_{n_i}]}).
$$
Recall that $k_{j}^i> 1$ for $i=1,\ldots,s ; j=1,\ldots,n_i$.

$\it 3\Rightarrow 4$. Using the index theorem we have
$$
\frac{1}{[k^i_{n_i}, k^i_{n_{i}-1}, \ldots, k^i_{1}]}+
\frac{1}{[l^i_{m_i}, l^i_{m_i-1}, \ldots,l^i_{1}]}=1.
$$
Notice that the order of the continued fraction is the inverse of
the one we need. However we have that:
if
$$
-k, -k_{1}^i,-k_{2}^i,\ldots, -k^i_{n_i},-1, -l^i_{m_i},-l^i_{m_i-1},
\ldots,-l^i_{1},k-1
$$
is  obtained by blow-ups as we explained then
$$
-k, -k_{n_i}^i,-k_{n_i-1}^i,\ldots, -k^i_1,-1, -l^i_1,-l^i_2,
\ldots,-l^i_{m_i},k-1
$$
is also obtained by blow-ups. This can be proved by induction on
the number of blow-ups. Notice that to create each branch of the
star we have done only one blow-up centered at a point of
$\sigma_\infty$ (the first blow-up) and so after obtaining the
desired star the self intersection of $\sigma_0$ is $k-s$.

$\it 4\Rightarrow 1$. We are looking for natural
numbers $a$ and $a_j^{i},\ j=1,2,\ldots, m_i,\ i=1,2,\ldots,s$ such
that the normal bundle of $\tilde Y=a\sigma_\infty+
\sum_{i=1}^s\sum_{j=1}^{m_i} a_{j}^i\tau_j^i$ is ample, i.e
$\tilde Y\cdot \sigma>0$ for $\sigma=\sigma_\infty$ and all
$\sigma_{j}^i$. These inequalities are translated into:
$$
-l^{i}_ja^i_j+a^i_{j-1}+a^i_{j+1}>0, a_{0}^i:=n, \ a_{m_i+1}^i:=0,
$$
$$
a(k-s)+\sum_{i=1}^s a_{m_i}^i>0.
$$
We rewrite these inequalities in the following way:
$$
\frac{a}{a_{1}^i}>[l_{1}^i, \frac{a_{1}^i}{a_{2}^i}]>
\ldots>
[l_{1}^i, l_{2}^i,\ldots, l_{m_i-1}^i, \frac{a_{m_i-1}^i}{a_{m_i}^i}]>
[l_{1}^i, l_{2}^i,\ldots, l_{m_i-1}^i, l_{m_i}^i],
$$
$$
\sum_{i=1}^s \frac{1}{\frac{a}{a_{1}^i}}>s-k.
$$
The existence of positive rational numbers $\frac{a^i_j}{a_{j-1}^i}$
follows from the hypothesis 4. Notice that $l^i_j$ are all greater
than $1$ and so the $[l_{1}^i, l_{2}^i,\ldots, l_{m_i-1}^i, l_{m_i}^i]$'s
are positive.
\end{proof}

We denote by $\mathcal{V}$  the variety obtained by the blow down of the
divisor $D_0$ in $\tilde{\mathcal{V}}$. We also denote by $\lambda$ the $\bc^*$-action
on $\mathcal V$ corresponding to $\tilde \lambda$ in $\tilde{\mathcal{V}}$.
\begin{Proposition}
\label{affine}
The variety  $\mathcal{V}$ is affine algebraic
and the $\bc^*$- action $\lambda$ is given
by a good action in some affine coordinates.
\end{Proposition}
\begin{proof}
Since the self intersection matrix of $D_0$ is negative definite,
by Lemma \ref{23.4.07}  we have that $D_\infty$ is
the  support of a divisor $Y$ with positive normal bundle.
By \cite{Hartshorne} Theorem 4.2 there exists a birational morphism
$f:X\rightarrow \tilde X\subset \mathbb P^\nu$ such that $f$ is an
isomorphism in a Zariski open neighborhood of
$D_\infty$ and $af(Y)$
for some big positive integer  $a$ is a hyperplane section.
We have $f=[f_0:f_1:\ldots :f_\nu]$, where $f_0,f_1,\ldots, f_\nu $ is
a $\bc$-basis of $H^0(X, \mathcal{O}_X(aY))$ for $a>0$ big enough.
Here $\mathcal{O}_X(aY)$ is the sheaf of meromorphic functions $u$
on $X$ with $div(u)+aY>0$. Since $\mathbb C^*$ acts on $H^0(X,
\mathcal{O}_X(aY))$ we can take
$f_i$'s  such that
$f_i(\lambda (x,t))=t^{q_i}f_i(x)$ for some $q_i\in \mathbb{N}$.
It turns out that $f$ is an isomorphism in $X-D_0$ and the divisor
$D_0$ is mapped to a point of $p\in \mathcal{V}$.
\end{proof}


\subsection{Existence of a global linearization}
We introduce the \emph{attraction basin} $B_p$ of $p$, by the flow
$\vr$, as
$$
B_p=\{\vr(t,z); t\in \mathbb{C^*};
z\in U \},
$$
where $U\subset V$ is the image of a neighborhood  $\tilde U$ of $\sigma_0$ in
$\tilde V$ by the resolution map $\rho$. The theorem of Suzuki([25])
asserts that the foliation $\fa_\vr$ admits a meromorphic first
integral. This implies that the singularities of $\tilde\fa_\vr$ are
linearizable and together with Theorem \ref{Theorem:desingularization}
that $B_p$ contains an open neighborhood of $p$. This fact
will be proved again during the construction of the conjugacy map between
$\bc^*$-actions. We aim to construct a conjugacy between $\vr$ on $B_p$
and $\lambda$ on $\mathcal V$ establishing the following theorem:

\begin{Theorem}
\label{Theorem:conjugacybasins}
 The set $B_p$ is an open subset of $V$ and there is a biholomorphism
$h: B_{p}\to \mathcal V$ which is a conjugacy between the actions
$\vr$ and $\lambda$, i.e.,
$$
h(\vr(t,z))=\lambda(t,h(z)),
\text{ for every } (t,z)\in \bc^*\times B_{p}.
$$
 \end{Theorem}

\begin{proof}

It will be enough to show that there is a conjugacy between
$\tilde \vr$ on $\tilde B_p:= \rho^{-1}(B_{p})$ and $\tilde \lambda$
on $\tilde {\mathcal V}$. We start by defining the conjugacy in a
neighborhood of $\sigma_0$. An immediate  consequence of Theorem
\ref{mai} is that there is a biholomorphic conjugacy
$h\colon \tilde U  \to \tilde {\mathcal U}$ between the restrictions
of $\tilde \vr$ and $\tilde \lambda$, where $\tilde U$
is a neighborhood of $\sigma_0$ in $\tilde V$ and $\tilde {\mathcal U}$
is a neighborhood of $\sigma_0$ in $\tilde {\mathcal V}$.
The conjugacy $h$ extends along the flows $\tilde \vr$ and $\tilde
\lambda$ as follows: For a point $z'\in \tilde B_p\backslash D$ there
is $t\in \bc^*$ with such that $z:=\vr(t,z')\in \tilde U$.
We define $h(z')$ by the equality $h(z')=\tilde \lambda(t^{-1},h(z))$.
It remains to extend $h$ to a neighborhood of the invariant manifolds
of the fixed points of $\tilde \vr$ in $\bigcup_ {i=1}^{r}\sigma_{i}$.
These points are all simple and lie in the linear chains starting at
$r_1,...,r_s$ in $\sigma_0$. Fix the linear chain starting at
$r_1$=$p_0$. The linear chain consists of a finite sequence of elements
of the divisor $\sigma_{0}, \sigma_{1},..., \sigma_{n}$ such that
$\sigma_{i}$, for $i\ne 0$ is a Riemann sphere, where the action
$\varphi_{1}$ is nontrivial, and $\sigma_{i}\cap\sigma_{i+1}= \{p_{i}\}$
is a simple singularity of $\tilde \vr$ for $i=1,..., n-1$. Since at each
$\sigma_{i}$, $i>0$, $\tilde \vr$ has two singularities, there is another
fixed point of $\tilde \vr$, $p_{n}\in \sigma_{n}$. The conjugacy $h$ is
already defined on $\sigma_{1}\backslash\{p_1\}$. The next lemma will
imply that $h$ extends to $\sigma_{2}\backslash \{p_2\}$. Proceeding by
induction and having already extended $h$ to $\sigma_{n}\backslash \{p_n\}$
the next lemma will apply again to extend $h$ to the remaining invariant
manifold of $p_{n}$. The same procedure can be followed on the other linear
chains starting at $r_1,..., r_s$ in $\sigma_0$. It only remains to prove
the following lemma.

\begin{Lemma}
\label{Lemma:siegelextension} Let $Z$ be a holomorphic
vector field defined in a neighborhood $\mathcal N$ of the origin $0\in
\bc^2$ and  $Z_1= nx\frac{\partial}{\partial x} - m y
\frac{\partial}{\partial y}$
with $n, m \in \mathbb N$ be its linear part. Suppose that
\begin{enumerate}
 \item
$x=0$ and $y=0$ are  separatrices of $Z$
\item There is an analytic conjugacy $h\colon \mathcal N\setminus
\{x=0\}\to \mathcal N\setminus \{x=0\}$ between $Z$ and $Z_1$, i.e.,
$h_*Z=Z_1$.
\end{enumerate}

Then $h$ extends to $\mathcal N$ as an analytic conjugacy between $Z$ and
$Z_1$.

\end{Lemma}
{\it Proof.}
The vector fields $Z$ and $Z_1$ are in the {\it Siegel domain}
(\cite{Arnold}), and the axes $\{x=0\}$ and $\{y=0\}$ are the only
local separatrices for $Z$ and $Z_1$. The conjugacy $h$ induces a
conjugacy between the  holonomies of the local separatrices $\{y=0\}$
and therefore by classical arguments (\cite{Mattei-Moussu},
\cite{Martinet-Ramis}) the foliations induced by the vector fields
$Z$ and $Z_1$ are analytically equivalent. Let $F\colon \mathcal N \to
\mathcal N$ be a biholomorphism such that $F_* Z =Z_1$. Then the map
$G=h\circ F^{-1}\colon \mathcal N\setminus \{x=0\} \to \mathcal N
\setminus \{x=0\}$
is a biholomorphism such  that $G_*Z_1=Z_1$. It is then enough to
show that such a self-conjugacy for $Z_1$ extends as a holomorphic
self-conjugacy to $\mathcal N$. This is proved as follows.
Write $G(x,y)=(x u, y v)$ for some holomorphic functions $u(x,y),
v(x,y)$ in $\mathcal N\backslash\{x=0\}$. From $G_*Z_1=Z_1$  we obtain that:
\[
nxu_x - myu_y=0,\, \,  \, \\
nxv_x - myv_y=0 \ \ \ \ (*)
\]
 Since $G$ is holomorphic in $\{y=0\}\setminus \{0\}$ we can write
in Laurent series
\[
u=\sum\limits_{i\in\mathbb Z \, j\in \mathbb N} u_{ij} x^i y^j \,
\, , \, \, v=\sum\limits_{i\in\mathbb Z \, j\in \mathbb N} v_{ij}
x^i y^j
\]
>From the above relations (*) we obtain
\[
(ni-mj)u_{ij}=0, \, (ni-mj)v_{ij}=0, \, \forall (i,j)\ne(0,0).
\]
Thus, $u_{ij}=0$ and $v_{ij}=0$ if $ni-mj\ne 0$. On the other
hand, if $ni-mj=0$ then $ni=mj\geq 1$ and therefore $i\geq 1$.
This shows that the Laurent series above only have positive powers.
Therefore $G$ extends as a holomorphic map to the axis $\{x=0\}$.
Since the same argument applies to $G^{-1}$ we conclude that $G$
extends to $\mathcal N$ as a biholomorphism preserving $Z_1$.
\end{proof}

\begin{Remark}\rm
 We observe that, as a consequence of Theorem
\ref{Theorem:desingularization}
and Lemma \ref{Lemma:siegelextension}
the singularity $p\in V$ is \emph{absolutely dicritical} in the
sense that there is a neighborhood $W$ of $p$ in $V$ such that
every leaf of $\fa$ intersecting $W$ contains a separatrix of
$\fa$ through $p$. In other words, for every leaf $L$ of the
restriction $\fa\big|_W$ the union $L\cup\{p\}$ is a separatrix
of $\fa$ through $p$.
\end{Remark}

\section{Basins of attraction of dicritical singularities}
\label{Section:basins}

The main result of this section is the following.

\begin{Theorem}
\label{Theorem:basinprop} Let $\vr$ be a holomorphic action of  $\bc^*$
on a normal Stein  space $V$ of dimension two. If $p\in V$ is a dicritical
singularity of $\fa_\vr$ then the attraction basin of $p$ is $V$. In other
words, every orbit of $\vr$ on $V\setminus \{p\}$ accumulates on $p$.
\end{Theorem}

This theorem will follow from the two lemmas below. A $\bc$-action is of type
$\bc^*$ if its generic orbit is biholomorphic to $\bc^*$.

\begin{Lemma}
\label{Lemma:analyticboundary} Let $\vr$ be a holomorphic $\bc$-action
of type $\bc^*$ on a normal Stein space of dimension two $V$.
Suppose that the set of fixed points of $\vr$ is discrete and
let $p\in V$ be dicritical singularity of $\fa_\vr$. Then the boundary $\po B_{p}$ of the
basin of attraction $B_{p}\subset V $ of $p$ is a
{\rm(}possibly empty{\rm)} union of analytic curves, and each one
of these curves accumulates at a nondicritical singularity of $\vr$.
\end{Lemma}
{\it Proof.} Suppose $\po B_{p}$ is nonempty. Then it is
invariant by $\fa_{\vr} $, i.e.  it is a union of leaves of $\fa_{\vr}$
and fixed points of $\vr$. We divide the argument in two steps.

\vglue.1in \noindent{\bf Step 1}:\, $\po B_{p}$ contains no closed
leaf.
\begin{proof}[Proof of Step 1]
Suppose that $L_0 \subset \po B_{p}$ is a leaf of $\fa_{\vr} $. Since
$\fa_{\vr} $ admits a meromorphic first integral, either $L_0$ is
closed in $V$ or it accumulates only on singular points.
Suppose that $L_0$ is closed in $V$ then it is an analytic smooth
curve in $V$. Since $V$ is Stein there is a holomorphic function
$h\in \mathcal O(V)$ such that  $\{h=0\}=L_0$ in $V$
(\cite{Gunning III}, Theorem 5, p.99). Since $L_0$ is a real
surface diffeomorphic to a cylinder $S^1 \times \re$, we
can take a generator $\ga\colon S^1 \to L_0$ of the homology of
$L_0$ and a holomorphic one-form $\al$ on $L_0$ such that $
\int_\ga \al = 1$. Again because  $V$ is Stein by Cartan's lemma
there is a holomorphic one-form $\tilde\al$ on $V$ which extends
$\al$. Since $\fa_{\vr} $ has a meromorphic first integral on $V$ then
the holonomy of $L_0$ is finite, say of order $n$. Let $\Sigma$ be
a small disc transverse to the leaves of $\fa_{\vr} $ with
$\Sigma\cap L_0 = \{p_0\} \in \ga(S^1)$.
Then there is a fixed power $\ga_{p_0}$ of $\ga$ which has
closed lifts $\tilde\ga_z$ to the leaves $L_z$ of $\fa_{\vr} $ that
contain the points $z \in \Sigma$. Thus, for $z \in \Sigma$ close
enough to $p_0$ we have $\big| \int_{\tilde\ga_z} \tilde\al
-\int_{\ga_{p_0}}\tilde \al \big| < \frac 12$\,, but
$\ga_{p_0} = n\ga$ and, since $\ga \subset L_0$\,, $
\int_{\ga_{p_0}} \tilde\al = n$ so that $ \int_{\tilde\ga_z}
\tilde\al \ne 0$. On the other hand $\tilde\al$ is holomorphic so
that $\tilde\al\big|_{L_z}$ is holomorphic and therefore closed
what implies, since $\tilde\ga_z \subset L_z$ is closed, that
$L_z$ has nontrivial homology and therefore necessarily $\ov L_z
\cong\bc^*$. However, since $L_0 \subset \po B_{p}$ there are
leaves $L_{z}$ of $\fa $ with $z\in \Sigma$ as above and which
satisfy $L_{z} \subset B_{p}$\,. Such a leaf $L_{z}$
accumulates on $p$ and therefore $L_{z} \cup \{p\}$ is a holomorphic
curve biholomorphic to $\bc$ and thus with trivial homology,
yielding a contradiction.
\end{proof}

\noindent{\bf Step 2}:\, $\po B_{p}$ contains no isolated
singularity, thus it is a union of analytic curves. Each
one of these curves contains a nondicritical fixed point
of $\vr$.
\begin{proof}[Proof of Step 2]
If $\po B_{p}$ contains an isolated point $P$ then
$\po B_{p}=\{P\}$ and $V=B_{p}\cup \{P\}$
would be compact contradicting the fact that $V$ is Stein.
On the other hand, by the first step each leaf $L$ contained in
$\po B_{p}$ is not closed in $V$ so that it accumulates at some
fixed point $P$ of $\vr$ and since
$\ov L \supset L \cup \{P\} \simeq \bc^* \cup \{0\} =
\bc$ and $ \ov L$ cannot be compact, it follows that
$\ov L =L \cup \{P\}$ is an analytic curve in $V$ and $L$ accumulates at
no other fixed point of $\vr$.
Finally, we observe that if a leaf $L \subset \po B_{p}$ accumulates at a
fixed point $P$ of $\vr$ then this singularity
is nondicritical: the basin of attraction of a dicritical singularity
in a Stein variety is open and contains an open neighborhood
of the singularity. Since $P \in \po B_{p}$\,, then some
leaf $L_1 \subset B_{p}$ intersects this neighborhood and therefore
$L_1$ accumulates on both  $p$ and $P$. Such a leaf would be
contained in a rational curve in $V$ and this is not possible because
$V$ is Stein. Thus $P$ is nondicritical.
\end{proof}

\begin{Lemma} \label{Lemma:biholomorphictoplane}  Let $\vr$ be
a holomorphic $\bc$-action  of type $\bc^*$ with a discrete set of fixed points
on a  normal Stein space of dimension two $V$  and an
absolutely dicritical singularity at $p\in V$.  Assume that
$H^1(B_{p},\mathbb R)=0$ {\rm(}for instance, if $B_{p}$ is
simply-connected{\rm)}. Then $V = B_{p}$.
\end{Lemma}

\begin{proof}  Let us first prove that $V=B_{p}\cup \po B_{p}$.
Indeed, put $A = V-\po B_{p}$ and $B = B_{p}$\,. Then $A$ is a
connected open subset of $V$ because by
Lemma~\ref{Lemma:analyticboundary}, $\po B_{p}$ is a thin set and
therefore it does not disconnect $V$. Since $B$ is also open and
connected and $\po A = \po B$ it follows that $A=B$ because $B
\subset A$. Therefore $V = B_{p} \cup \po B_{p}$\,. Let us now
prove that
 $\po B_{p} = \emptyset$. Suppose
there is some analytic curve (leaf) $L_0 \subset \po B_{p}$\,.
Again, since  $V$ is Stein there is a holomorphic function $f\in
\mathcal O(V)$ such that  $\{f=0\}=\ov{L_0}$ in $V$ (\cite{Gunning
III}, Theorem 5, p.99). Define the meromorphic one-form $\al =
\frac{df}{f}$ on $V$, the polar set of $\alpha$ is $\ov{L_0}$.
Given a disc $\Sigma \cong \bd$ transverse to the leaves of
$\fa_{\vr} $
with $\Sigma\cap L_0 = \{p_0\}$ we consider a simple loop $\ga\colon
S^1 \to\Sigma$ around $p_0 \in \Sigma$ such that $ \int_\ga \al =
2\pi\sqrt{-1}$.
 We can assume that $\ga(S^1) \subset B_{p}$ because
$\Sigma\setminus (\Sigma \cap \po B_{p})\subset B_{p}$ and $\po
B_{p}$ is thin. Since by hypothesis $H^1(B_{p},\mathbb R)=0$ it
follows that  $ \int_\ga \al = 0$ yielding a contradiction.
Therefore $\partial B_p$ contains no leaf. Thus $V = B_{p}$ and
the Lemma is proved.
\end{proof}

\begin{proof}[Proof of Theorem~\ref{Theorem:basinprop}]
Since $V$ is Stein, all regular leaves of $\fa_\vr$
are biholomorphic to $\bc^*$. By Proposition 5 and Theorem 2 of \cite{Suzuki} and the fact that
$\fa_\vr$ has a dicritical singularity, $\vr$
has isolated fixed points. In view of
Lemmas~\ref{Lemma:analyticboundary} and
\ref{Lemma:biholomorphictoplane}
it is enough to observe that $H^1(B_{p},\mathbb R)=0$.
This is clear since  the basin $B_p$ of
the action $\vr$  is biholomorphic to the basin
$\mathcal V$ of the linear periodic flow $\lambda$ on $\mathcal
V$, and moreover $H^1(\mathcal V,\mathbb R)=0$.
\end{proof}

\begin{Remark}
\label{Remark:uniquedicritical} {\rm The proof of
Theorem~\ref{Theorem:basinprop} also shows that
$V\setminus\{p\}$ contains no singularity of $\vr$ which is
dicritical as a singularity of $\fa_\vr$.}
\end{Remark}

\begin{proof}
Suppose by contradiction that  $q$ is a dicritical singularity of
$\vr$ then we consider the attraction basin $B_q$ of $q$  and
proceeding as for $p$ we prove that $V\setminus\partial B_q= B_q$.
On the other hand, since $\partial B_q$ is a thin set we have that
$V\setminus \partial B_q$ is connected. Clearly we have $\partial
B_p \cap B_q = \emptyset$ because otherwise, since $B_q$ is open,
there would be orbits contained in $B_p$ and $B_q$, which is not
possible because these orbits would be contained in rational curves.
Analogously we have $\partial B_q \cap B_p = \emptyset$. Thus, the
only possibility is to have  $\partial B_p=\partial B_q$.
Therefore, $V\setminus \partial B_p=V\setminus \partial B_q$, i.e,
$B_p=B_q$ and this gives a contradiction.

\end{proof}

\section{Proof of the main theorems}
\label{Section:proofmain}
\begin{proof}[Proof of Theorem \ref{Theorem:main}]
Let us be given a pair $(V,\vr)$ as in Theorem \ref{Theorem:main}.
By Theorem \ref{Theorem:basinprop}  the basin of attraction of
the dicritical singularity $p$ is the whole space $V$, i.e $B_{p}= V$.
By Theorem \ref{Theorem:conjugacybasins} there is a biholomorphic
conjugacy $h\colon V \to \mathcal V$ between $\vr$ and $\lambda$.
Finally, by Proposition \ref{affine} the variety
$\mathcal V$ is affine and the the action $\lambda$ of $\bc^*$
on $\mathcal V$ in some affine coordinates is good.
Note that our proof gives an alternative proof of
Proposition 1.1.3, page 207 in \cite{Orlik}.
\end{proof}

\begin{proof}[Proof of Theorem \ref{Theorem:main2}]
Each data in Theorem \ref{Theorem:main2} gives us a linear model variety
$(\mathcal{V},\lambda)$ and in  a similar way as in Theorem
\ref{Theorem:conjugacybasins} we can prove that two linear models are
biholomorphic if and only their correspondings data are the same.
On the other hand, we have proved that each pair $(V,\vr)$  as in Theorem
\ref{Theorem:main}  is biholomorphic to
a linear model variety $(\mathcal V,\lambda)$.
\end{proof}

\begin{proof}[Proof of Corollary~\ref{Corollary:smooth}]
Since $V$ is smooth, the resolution process of $p\in V$ is the
blow-up resolution (\cite{Seidenberg}) for the foliation $\fa_\vr$
at $p$. In particular, $\sigma_0$ is a negatively embedded
projective line. Thus, by Theorem~\ref{Theorem:main} there is a
good action $\sigma_Q$ on $\mathcal V\subset \bc^{n+1}$ equivalent
to $\vr$ on $V$. Thus $\mathcal V$ is a quasi-homogeneous
non-singular algebraic surface on $\bc^{n+1}$ and therefore it is
a graph, hence equivalent to an affine plane by algebraic change
of coordinates.
\end{proof}

\begin{Remark}
\label{Remark:quasihomogeneous} {\rm A {\it quasi-homogeneous
surface singularity} (see for instance \cite{Seade} Chapter III,
page 67) is a $2$-dimensional analytic variety $V\subset \bc^m$
with an isolated singularity at $0\in \bc^m$ supporting a
$\bc^*$-action $\vr$ which is {\it good} in the sense that every
non-singular orbit accumulates (only) at $0\in\bc^m$.
As a consequence of our Theorem~\ref{Theorem:main} we obtain that
if $V$ is a two-dimensional Stein space with a $\bc^*$-action
having a dicritical singularity at $p\in V$ then $p\in V$ is a
quasi-homogeneous surface singularity. }
\end{Remark}
\bibliographystyle{amsalpha}

\begin{tabular}{ll}
C. Camacho, H. Movasati  & \qquad  B. Sc\'ardua\\
IMPA-Estrada D. Castorina, 110 & \qquad Instituto de Matem\'atica\\
Jardim Bot\^anico  & \qquad Universidade Federal do Rio de Janeiro\\
Rio de Janeiro - RJ  & \qquad  Caixa Postal 68530\\
CEP. 22460-320   & \qquad 21.945-970 Rio de Janeiro-RJ\\
BRAZIL &  \qquad BRAZIL
\end{tabular}

\end{document}